\documentclass{conm-p-l}

\usepackage{amssymb}

\newtheorem{theorem}{Theorem}[section]

\newtheorem{proposition}[theorem]{Proposition}

\theoremstyle{definition}
\newtheorem{definition}[theorem]{Definition}
\newtheorem{example}[theorem]{Example}

\numberwithin{equation}{section}

\newcommand\C{\mathbf{C}} 

\newcommand\Q{\mathbf{Q}}
\newcommand\Z{\mathbf{Z}}
\newcommand\F{\mathbf{F}}
\newcommand\twomatr[4]{\begin{pmatrix}#1&#2\\ #3&#4 \end{pmatrix}}
\newcommand\stwomatr[4]{\left(\begin{smallmatrix}#1&#2\\  
                               #3&#4 \end{smallmatrix}\right)}

\newcommand\tensor{\otimes}
\newcommand\isomorphic{\cong}

\newcommand\union{\cup}

\newcommand\intersect{\cap}

\newcommand\abs[1]{{\left|#1\right|}}

\newcommand\Projective{{\bf P}} 
\newcommand\modforms{\mathcal{M}}
\newcommand\cuspforms{\mathcal{S}}

\newcommand\Half{\mathcal{H}} 
\newcommand\GammaN{{\Gamma(N)}} 
\newcommand\LL{\mathcal{L}} 
\DeclareMathOperator{\Divisor}{div}
\newcommand\phee{\varphi}

\begin{document}

\title
[Modular forms constructed from moduli; explicit modular curves]
{Modular forms constructed from moduli of elliptic curves, with
  applications to explicit models of modular curves}


\author{Kamal Khuri-Makdisi}
\address{Mathematics Department, 
American University of Beirut, Bliss Street, Beirut, Lebanon}
\email{kmakdisi@aub.edu.lb}
\thanks{September 1, 2018}




\begin{abstract}
These are the lecture notes from my portion of a mini-course for
the summer school ``Building Bridges 3'' that was held in Sarajevo
during July 2016.  My lectures covered the Katz definition of modular
forms, a family of forms defined from this perspective and their
relation to Eisenstein series, and methods of finding explicit models of
modular curves.  The treatment is purely expository, and the results are
mostly standard, although a few points of view may not be as widely
known as they deserve to be.
\end{abstract}

\maketitle


\section{Lecture 1}

In the previous lectures in this summer school, we have considered
modular forms as holomorphic functions $f(\tau)$ for $\tau \in \Half$,
with the $q$-expansion (when $f$ is a newform) encoding an associated
Galois representation that we have used as a black box.

We now want to describe the connection between modular forms and the
modular curves, such as $X(N)$, parametrizing elliptic curves with level
structure.  A word of caution: this parametrization of elliptic curves by
the points of the modular curve $X(N)$ is completely different from the
arithmetic parametrization of a single elliptic curve over $\Q$ as a
quotient of the Jacobian of $X_0(N)$.

Recall our fundamental congruence subgroups of $\operatorname{SL}(2,\Z)$:
\begin{equation}
\label{equation1.1}
\begin{split}
\operatorname{SL}(2,\Z) &= \Gamma(1) > \Gamma_0(N) > \Gamma_1(N) > \Gamma(N),\\
\Gamma_0(N) &=
       \{\twomatr{a}{b}{c}{d} \in \Gamma(1) \mid c \equiv 0 \pmod{N}\},\\
\Gamma_1(N) &= 
       \{\twomatr{a}{b}{c}{d} \mid c \equiv 0,\quad
                                   a \equiv d \equiv 1 \pmod{N}\},\\
\Gamma(N) &=
       \{\twomatr{a}{b}{c}{d} \mid b \equiv c \equiv 0,\quad
                                   a \equiv d \equiv 1 \pmod{N}\}.\\
\end{split}
\end{equation}

The connection between modular forms and elliptic curves arises by
associating, to each value $\tau \in \Half$,  an elliptic curve
$E_\tau$, which is analytically $E_\tau = \C/(\Z+\Z\tau)$.  Suppose
$\tau, \tau' \in \Half$ are related via an element $\gamma =
\stwomatr{a}{b}{c}{d} \in \Gamma(1)$, so that 
$\tau' = \gamma \tau = (a\tau+b)/(c\tau+d)$.  It then follows that
$E_\tau \isomorphic E_{\tau'}$, via
\begin{equation}
\label{equation1.2}
 z + \Z + \Z\tau \in E_\tau
   \longleftrightarrow
 z' = \frac{z}{c\tau + d} + \Z + \Z\tau' \in E_{\tau'}.
\end{equation}
This basically expresses, in terms of the elliptic curves $E_\tau$ and
$E_{\tau'}$, the standard fact that the lattices
$\Z + \Z\tau$ and $\Z + \Z\tau'$ are homothetic if and only if $\tau'$
and $\tau$ are related by an element $\gamma \in \Gamma(1)$.  For
equivalence under the smaller congruence subgroup
$\Gamma(N)$, we have the following more precise statement.

\begin{proposition}
\label{proposition1.1}
Let $\tau, \tau' \in \Half$.  Then $\tau$ and $\tau'$ are related by
an element $\gamma \in \GammaN$ if and only if there is an isomorphism
$\phi: E_\tau \to E_{\tau'}$ between the corresponding elliptic
curves, such that 
\begin{equation}
\label{equation1.3}
\phi(1/N) = 1/N, \qquad \phi(\tau/N) = \tau'/N.
\end{equation}
The above is of course shorthand for saying
$\phi(1/N + \Z + \Z \tau) = 1/N + \Z + \Z\tau'$ 
and  $\phi(\tau/N + \Z + \Z \tau) = \tau'/N + \Z + \Z\tau'$.
\end{proposition}

From the above, we deduce that the quotients 
$Y(1) = \Gamma(1) \backslash \Half$ and $Y(N) = \GammaN \backslash
\Half$ parametrize, respectively, isomorphism classes of elliptic curves
over $\C$, and isomorphism classes of triples $(E,P,Q)$ where $E$ is
an elliptic curve over $\C$, and $\{P,Q\}$ is a basis for the
$N$-torsion $E[N]$ that is symplectic, in the sense that $e_N(P,Q) =
\exp(2\pi i/N)$ for the Weil pairing.  The $\GammaN$-orbit of a point
$\tau \in \Half$ corresponds to the isomorphism class of the triple
$(\C/(\Z+\Z\tau), 1/N, \tau/N)$.  

One says that $Y(1)$ and $Y(N)$ are \emph{moduli spaces} parametrizing
the moduli of elliptic curves, respectively without or with a basis
for the $N$-torsion.  We leave it to the reader to look up or
determine as an exercise the moduli problems that are parametrized by
the quotients $Y_0(N) = \Gamma_0(N)\backslash \Half$ and
$Y_1(N) = \Gamma_1(N)\backslash \Half$.

It turns out to be much better to work with a compactification, the
modular curve $X(N)$, of $Y(N)$, which one can think of as adding the
cusps to $Y(N)$\footnote{
There is an extensive theory of these models of
modular curves, not just over $\C$, but over number fields and in
characteristic $p$.  We will not have the space to touch directly on the
arithmetic aspects in these lectures, but the reader is encouraged to think
at least about the way in which our discussion over $\C$ in fact takes
place over a number field (which is usually $\Q$ or the $N$th cyclotomic
field), viewed as a subfield of $\C$.
}.
With the cusps included, one can view the space of modular forms
$\modforms_k(\GammaN) = H^0(X(N), \LL^k)$ as being the space of holomorphic
sections of the $k$th power of a line bundle $\LL$ on $X(N)$, at least for
$N \geq 3$ to avoid issues of elliptic elements (which would arise if we
used, say, $X_0(N)$).  This point of view makes the required behavior of
modular forms at cusps automatic, once we require holomorphy at the
cusps that have been added to obtain $X(N)$.  We will follow up on the
interpretation of modular forms as sections of line bundles in Lecture~2.

In this lecture, we will focus instead on interpreting modular forms
based on the moduli of elliptic curves parametrized by $X(N)$.  The
precise algebraic formulation is due to N. Katz; see for example
Section~2.1 of~\cite{Katz}.
We shall be somewhat cavalier with the precise definition, and simply
state the following.

\begin{definition}
\label{definition1.2}
Let $k \geq 1$ be an integer.
A Katz modular form of weight~$k$ on $\GammaN$ is a ``nice'' function
$f(E,P,Q,\omega)$ satisfying the homogeneity property
\begin{equation}
\label{equation1.4}
f(E,P,Q,c \omega) = c^{-k} f(E,P,Q,\omega).
\end{equation}
Informally, the domain of definition of $f$ is tuples
$(E,P,Q,\omega)$, where $E$ is an elliptic curve, the pair $(P,Q)$ is
a symplectic basis for the  $N$-torsion $E[N]$, and $\omega \in
H^0(E,\Omega^1)$ is a global $1$-form on $E$.
The precise definition allows the argument $E$ in the tuple to be a
generalized elliptic curve scheme $E/S$ over a base $S$, and brings in
compatibility conditions under change of base; over $\C$, these
compatibility conditions amount
to holomorphy on $\Half$ and at the cusps (this is what we mean by
a ``nice'' function).  In these lectures, however, we will pretend
to consider only tuples defined over $\C$.  In this context, the
choice of $\omega$, up to a complex scalar, determines the homothety
class of the lattice $L$ of periods of $E$, namely $L = \{\int_\gamma
\omega \mid \gamma \in H_1(E,\Z)\}$, and $E \isomorphic \C/L$.  This
ties in with the perspective seen elsewhere, of modular forms on
$\Gamma(1)$ as functions of lattices.  In the context of these
lectures, we will simply pass between modular forms $f(\tau)$ with
$\tau \in \Half$, and the corresponding Katz modular form which we
evaluate on the tuple 
$(E,P,Q,\omega) = (\C/(\Z+\Z\tau), 1/N, \tau/N, dz)$;
here $z$ is the coordinate on $\C$ when we view $E = \C/(\Z+\Z\tau)$.
\end{definition}

\begin{example}
\label{example1.3}
In case of level $N=1$, we can dispense with specifying the $P$ and
$Q$ in the tuples above.  Then the Eisenstein series (restricted in this
example to even weight $k \geq 4$) is given in Katz and traditional form as
\begin{equation}
\label{equation1.5}
\begin{split}
G_k(E,\omega) &= \sum_{0 \neq \gamma \in H_1(E,\Z)} \left[ \int_\gamma
  \omega \right]^{-k},\\
G_k(\tau) &= \sum_{0 \neq m + n\tau \in \Z + \Z\tau} (m+n\tau)^{-k}
 = {\sum}' (m+n\tau)^{-k},\\
\end{split}
\end{equation}
where as usual ${\sum}'$ denotes a sum where we omit terms that look
like $0^{-k}$ and are hence meaningless.

We admit that calling~\eqref{equation1.5} a Katz form is not quite
fair, because one really wants to have an algebraic construction of
the value on the tuple.  So we point out that $G_4$ and $G_6$ can be
obtained from the following algebraic construction: starting from the
elliptic curve $E$, its short Weierstrass form $y^2 = x^3 + ax + b$
(over, say, a field not of characteristic $2$ or~$3$) is uniquely
determined up to changing $(x,y)$ into $(x',y') = (c^{-2} x, c^{-3}
y)$ with $c \neq 0$; this transforms $(a,b)$ into $(a',b') = (c^{-4} a,
c^{-6} b)$.  However, if one has access not only to $E$ but also to the
global differential $\omega$ on $E$, then one can normalize the choice
of coordinates to obtain $\omega = dx/2y$.  In this context, the
coefficients $a$ and $b$ in the normalized Weierstrass equation are
Katz modular forms of level $\Gamma(1)$ of weights $4$ 
and~$6$, and are in fact simple multiples of $G_4$ and $G_6$,
respectively.  This can be seen via the parametrization of a complex
elliptic curve by $x = \wp(z) = z^{-2} + \cdots, y = (1/2)\wp'(x) =
-z^{-3} + \cdots$, and the usual differential equation for $\wp$.  We
can also identify the higher weight Eisenstein series on $\Gamma(1)$
as coefficients in the Laurent expansion of $\wp$, which can be
defined purely algebraically over a field of characteristic zero (in
terms of the completion of the local ring of $E$ at the origin $O$,
which allows us to integrate the formal power series of $\omega$ and
obtain an ``analytic'' uniformizer $z$ in this completed local ring).
\end{example}

We wish to generalize the abovementioned principle to $\GammaN$, and to
define a wide family of Katz-style modular forms, which when evaluated
on a tuple $(E,P,Q,\omega)$ are given as coefficients in
the Laurent series of certain elements in the function field of $E$,
which analytically can be viewed as elliptic functions with respect to
$\Z + \Z\tau$.  This family of modular forms will include all the
Eisenstein series on $\GammaN$.  We therefore begin by recalling the
definition of the relevant Eisenstein series.

\begin{definition}
\label{definition1.4}
For $i,j \in \Z$, let $\alpha = (i/N, j/N)$, which we will usually
view as an element of $\Q^2 / \Z^2$; sometimes, by abuse of
notation, we will identify $\alpha$ with the torsion point
$(i+j\tau)/N = iP + jQ \in E_\tau[N]$, all of which depends of course
on a varying $\tau \in \Half$, or equivalently on the corresponding tuple
$(E = E_\tau,P,Q,\omega)$.  We then define the Eisenstein series of
arbitrary weight~$k$, with parameter~$\alpha$, by
\begin{equation}
\label{equation1.6}
G_{k,\alpha}(\tau) 
  = {\sum_{m,n\in\Z}\!\!\!}' \bigr(m+n\tau + i/N + (j/N)\tau\bigr)^{-k}
  = {\sum_{\ell \in \Z + \Z\tau}\!\!\!\!}' (\ell + \alpha)^{-k}.
\end{equation}
In the above, the notation ${\sum}'$ means we omit the term with 
$\ell + \alpha = 0$, if it is present in the sum (which is essentially
only when $i=j=0$).  For $k > 2$, the sum in~\eqref{equation1.6} converges
absolutely and uniformly for $\tau$ in any compact set, and yields a
modular form of weight~$k$ on $\GammaN$. 

It is traditional to
modify the definition to make sense of the lower weights $k \in \{1,2\}$
by Hecke's summation method~\cite{Hecke27}:
\begin{equation}
\label{equation1.7}
\begin{split}
G_{k,\alpha}(\tau,s) 
 &= {\sum_{\ell \in \Z + \Z\tau}\!\!\!\!}' (\ell + \alpha)^{-k}
                                  \abs{\ell + \alpha}^{-2s},
\\
G_{k,\alpha}(\tau) &= G_{k,\alpha}(\tau,0) 
\text{ after analytic continuation in } s.\\
\end{split}
\end{equation}
This yields the same Eisenstein series as before for $k > 2$.  For $k = 2$,
it turns out that $G_{2,\alpha}(\tau)$ is not quite holomorphic, but is the
sum of $-\pi/(\text{Im } \tau)$ and a holomorphic function of $\tau$; so
to obtain a holomorphic weight~$2$ Eisenstein series one must consider a
difference such as $G_{2,\alpha} - G_{2,0}$.  Reassuringly, for
weight~$k=1$, $G_{1,\alpha}$ is indeed holomorphic.  We will not consider
the case $k=0$ in these lectures, but elsewhere in this summer school the
series $G_{0,0}(\tau,s)$ plays a major role via the Kronecker limit
formula.
\end{definition}

We need a few more preliminaries before we can construct the general
family of modular forms that we promised above.  In the meantime, to whet
the reader's appetite, let us give an \textit{ad hoc} Katz-style
interpretation of the Eisenstein series of weights $2$ and~$3$.  Take a
nonzero $\alpha$, and view it also as a nonzero $N$-torsion point on the
elliptic  curve $E_\tau$.  Then its coordinates on the Weierstrass model
are $(x_\alpha, y_\alpha) = (\wp(\alpha), (1/2)\wp'(\alpha))$, so we
obtain from the series for $\wp$ and $\wp'$ that
\begin{equation}
\label{equation1.7.5}
\begin{split}
\wp(\alpha) &= {\sum_{m,n}}'
     \left[ (\alpha+m+n\tau)^{-2} - (m+n\tau)^{-2} \right] 
     = G_{2,\alpha}(\tau) - G_{2,0}(\tau),\\
\wp'(\alpha) &= -2{\sum_{m,n}}' (\alpha+m+n\tau)^{-3}
       = -2 G_{3,\alpha}(\tau).\\
\end{split}
\end{equation}
This interprets the holomorphic weight~$2$ form $G_{2,\alpha} -
G_{2,0}$ as the $x$-coordinate of a torsion point on the Weierstrass
model (once normalized by the choice of global differential);
similarly, the form $G_{3,\alpha}$ is essentially the $y$-coordinate.
The identification $\wp(\alpha) = G_{2,\alpha}(\tau) - G_{2,0}(\tau)$
is however slightly more delicate than written above, since
convergence issues prevent us from simply expanding the sum over
$(m,n)$ in the first line above.  The end result is correct, however,
using techniques similar to the construction in
Definition~\ref{definition1.5} below.

It is also interesting to consider the slope
$\lambda = (y_\beta - y_\alpha)/(x_\beta - x_\alpha)$ through two
torsion points in the Weierstrass model; this has a natural
interpretation as a Katz modular form. 
Here we assume that $\alpha, \beta \neq 0$ and that $\alpha + \beta
\neq 0$.  The line in question through the points
$(x_\alpha,y_\alpha)$ and $(x_\beta,y_\beta)$ also passes through the
point $(x_\gamma,y_\gamma)$ on the elliptic curve, with
$\alpha+\beta+\gamma=0$, by the addition law on the elliptic curve.
It is immediate that this slope $\lambda$ is essentially the
ratio $(G_{3,\beta} - G_{3,\alpha})/(G_{2,\beta}-G_{2,\alpha})$, so it
transforms under $\GammaN$ the same way as a modular form of
weight~$1$.  The question is whether the quotient $\lambda$ (when
viewed as a function of $\tau$) is holomorphic or merely meromorphic on
$\Half$ and the cusps.  However, it is known from the formulas for the
addition law on a Weierstrass curve that $\lambda^2 = x_\alpha +
x_\beta + x_\gamma$, and this last sum is a genuine modular form
(being a linear combination of $G_2$'s), so $\lambda$ cannot have any
poles.  We will show later in this lecture that $\lambda$ itself is
essentially the sum $G_{1,\alpha}+G_{1,\beta}+G_{1,\gamma}$; this is
roughly equivalent to a classical formula for the Weierstrass $\zeta$
function, but we will present the argument differently below.

As our final preparatory comment on Eisenstein series, we point out
that one can bypass Hecke's analytic continuation in $s$ by adopting a
different point of view on Eisenstein series of weights $1$ and~$2$.
Instead of defining a single $G_{k,\alpha}$, it turns out that one can
write down convergent series for certain linear combinations
$\sum_\alpha m_\alpha G_{k,\alpha}$, with good convergence for all
$k \geq 1$.  The technique is as follows~\cite{KKMmoduli}.

\begin{definition}
\label{definition1.5}
Consider a finite number of $\alpha \in (1/N)\Z^2$, and attach to
each $\alpha$ a coefficient $m_\alpha \in \C$, with the properties
\begin{equation}
\label{equation1.8}
\sum_\alpha m_\alpha = 0, \qquad \sum_\alpha m_\alpha \alpha = 0.
\end{equation}
Write $D = \sum_\alpha m_\alpha [\alpha]$ for the formal linear combinations
of symbols $[\alpha]$.  We then define
\begin{equation}
\label{equation1.9}
G_{k,D}(\tau)
   = \left[ \sum_\alpha m_\alpha G_{k,\alpha}(\tau,s)\right]_{s=0}
   = \sum_{\ell \in \Z + \Z\tau}
       \left(
       {\sum_\alpha}' m_\alpha (\ell + \alpha)^{-k}
       \right).
\end{equation}
The latter sum converges in the given order $\sum_\ell (\sum_\alpha')$,
since condition~\eqref{equation1.8} implies that $\sum_\alpha'
m_\alpha (\ell + \alpha)^{-k} = O(\ell^{-k-2})$, which implies good
convergence of the sum over $\ell$ for all $k > 0$.  Note that if the
coefficients $m_\alpha$ are integers, we can view the formal sum $D$ as a
divisor on the elliptic curve $E_\tau$.  In that setting,
\eqref{equation1.8} says that $D$ is a principal divisor on $E_\tau$, and
that the preimages $\alpha$ of the points $\alpha + \Z + \Z\tau \in E_\tau$
are chosen so that their sum (in $\C$, not just in $\C/(\Z+\Z\tau)$) is
exactly zero.
\end{definition}

We can now describe the construction of our family of Katz modular forms.

\begin{definition}
\label{definition1.6}
Take a finite number of $\alpha \in N^{-1}\Z^2$ and coefficients
$m_\alpha \in \Z$, satisfying condition~\eqref{equation1.8} above.
Write $D = \sum_\alpha m_\alpha [\alpha]$ for the formal sum.
For $k \geq 1$, we define functions $f_{k,D}$ and $g_{k,D}$ of the
tuple $(E,P,Q,\omega)$ by the following procedure (which works only if
$E$ is defined over a field $K$ of characteristic zero):
\begin{enumerate}
\item
Associate to each $\alpha = (i/N,j/N)$ the point $P_\alpha = iP+jQ$ on the
elliptic curve $E$, as usual, and form the associated divisor (which
by abuse of notation will also be called $D$) as above.  Thus $D =
\sum_\alpha m_\alpha P_\alpha$, and we will ignore the distinction
between $\alpha$ and $P_\alpha$ whenever it suits us for the
exposition.  It follows that $D$ is a principal divisor on $E$, and
there exists an element $\phi_D \in K(E)$ (the function field of $E$)
with $\Divisor \phi_D = D$.  The values $f_{k,D}(E,P,Q,\omega)$ and
$g_{k,D}(E,P,Q,\omega)$ will be constructed out of the Laurent
expansion of $\phi_D$ and its logarithmic derivative $d\phi_D/\phi_D$
at the origin $O$ of the elliptic curve $E$.
\item
The Laurent expansion of $\phi_D$ at $O$ needs to be expressed in
terms of a uniformizer, i.e., a local coordinate near $O$ which
vanishes there. Since we are in characteristic zero, we can integrate
the global form $\omega$ to define an ``analytic uniformizer'' $z$ at
$O$, with $dz = \omega$.  More precisely, start with an ``algebraic
uniformizer'' $t \in K(E)$ at the origin $O$; for example, if a
Weierstrass form of $E$ is $y^2 = x^3 + ax + b$, then one choice of
algebraic uniformizer is $t = x/y$.  Then the completed local ring of
$E$ at $O$ is $\widehat{\mathcal{O}}_{E,O} = K[[t]]$.  In this
situation, the expansion of $\omega$ at $O$ can be written as
$\omega = (c_0 + c_1 t + c_2 t^2 + \cdots)dt$, with $c_0 \neq 0$.
Then define $z \in \widehat{\mathcal{O}}_{E,O}$ by
$z = \int \omega = c_0 t + c_1 t^2/2 + c_2 t^3/3 + \cdots$. 
It follows that in fact $\widehat{\mathcal{O}}_{E,O} = k[[z]]$ as well.
\item
Normalize $\phi_D$, which is unique up to a factor in $K^*$, by
requiring that its Laurent expansion at the origin $O$ be of the form
$\phi_D = z^n(1 + f_1 z + f_2 z^2 + f_3 z^3 + \cdots)$, where $n$ is
the multiplicity of $O$ in $D$ (this is the sum of $m_\alpha$ over all
those $\alpha$ that map to $O$ in the curve).  Similarly, and without need
for normalizing $\phi_D$, consider the logarithmic differential
$d\phi_D / \phi_D$, and its expansion in terms of $z$ at $O$:
$d\phi_D / \phi_D = (n/z + g_1 + g_2 z + g_3 z^2 + \cdots)dz$.
\item
The value of $f_{k,D}$ at our tuple is then the coefficient $f_k$
above, while the value of $g_{k,D}$ is the coefficient $g_k$ above.
\end{enumerate}
\end{definition}

\begin{proposition}
\label{proposition1.7}
The functions $f_{k,D}$ and $g_{k,D}$ defined above are Katz modular
forms, whose value on the standard tuple $(E_\tau, 1/N, \tau/N, dz)$
give modular forms in $\modforms_k(\GammaN)$.
Moreover, when we evaluate $g_{k,D}$ at the standard tuple
associated to $\tau \in \Half$, the value is
$g_{k,D}(E_\tau,1/N,\tau/N,dz) = -G_{k,D}$.
\end{proposition}
\begin{proof}
For the full proof, see Sections 2 and~3 of~\cite{KKMmoduli}.  We
note that in arbitrary characteristic, it is still possible to define
Katz modular forms from the Laurent series coefficients of $\Phi_D$ and its
logarithmic differential, when expanded in terms of the algebraic
uniformizer $t$, which can be chosen sufficiently canonically.

Let us sketch a proof of the assertion that $g_{k,D} = -G_{k,D}$, 
thereby exhibiting Eisenstein series as special cases of this
construction.  We first identify the function $\phi_D$ (up to a
constant factor, which disappears in the logarithmic differential).
Morally, we would like to write down directly the desired function
with zeros and poles as predicted by the (translates of the)
$\{\alpha\}$:
\begin{equation}
\label{equation1.10}
\phi_D(z) = \prod_{\ell \in \Z+\Z\tau} \left[
{\prod_\alpha}'
   \left(
       1 - \frac{z}{\ell+\alpha}
   \right)^{m_\alpha}
\right],
\end{equation}
where in the above, the notation $\prod'$ means that if $\ell+\alpha=0$,
then we include the factor $z^{m_\alpha}$ instead.
(The above product is somewhat different from the usual construction of
$\phi_D = \prod_\alpha \sigma(z-\alpha)^{m_\alpha}$ as a product of shifted
Weierstrass $\sigma$-functions.)

The question is whether the product in~\eqref{equation1.10} is really
doubly periodic with respect to $\Z + \Z\tau$.  The
conditions~\eqref{equation1.8} ensure that the above product over $\ell$
converges well.  Then the identity 
$1-\frac{z+w}{\ell+\alpha} =
  (1-\frac{w}{\ell+\alpha})(1-\frac{z}{\ell-w+\alpha})$, 
plus the good convergence of the products $\prod_\ell \prod_\alpha$
of each factor on the right hand side, tells us that $\phi_D(z+w) =
C_w \phi_D(z)$ for $w \in \Z + \Z\tau$, and some appropriate constant
$C_w$ for each $w$.  It turns out however that the $C_w$ are~$1$,
which follows by comparing the logarithmic differential of the product
in~\eqref{equation1.10}:
\begin{equation}
\label{equation1.11}
\frac{d\phi_D}{\phi_D} = 
\left[ \sum_{\ell \in \Z + \Z\tau} \left(\sum_\alpha
      \frac{m_\alpha}{z-\ell-\alpha}
\right)\right] dz,
\end{equation}
which agrees with the logarithmic derivative from the ``correct''
product of $\sigma$-functions, as presented in Theorem~2.8
of~\cite{KKMmoduli}.  Once again, the sum over $\ell$ has good
convergence, and we can expand $m_\alpha/(z-\ell-\alpha) =
-m_\alpha((\ell+\alpha)^{-1} + (\ell+\alpha)^{-2} z + \cdots)$ for the pairs
with $\ell+\alpha \neq 0$; the terms with $\ell+\alpha=0$ contribute
$n/z$.  Combining this, we get the expansion of $d\phi_D/\phi_D$ in
terms of the coefficients $\{g_k\}$, and we immediately identify each
$g_k$ as the negation of the desired Eisenstein series.
\end{proof}

\section{Lecture 2}

This lecture will discuss explicit models for modular curves.
At first, we will work primarily with the function field of $X_0(N)$
when we discuss the modular equation, but in the second part of the
lecture, we will view modular forms on $\GammaN$ primarily as sections
of line bundles on the modular curve $X(N)$.  We assume some familiarity with
Riemann-Roch spaces, but not necessarily with line bundles and their
connection with projective embeddings, which we will discuss
informally once we start using those concepts.

We will work entirely over $\C$, but, as in the first lecture, the
reader is encouraged to picture how most of our constructions actually
take place over a number field, viewed as a subfield of $\C$.  This is
the main, but not the only, source of the arithmetic subtlety captured
by modular forms.  The background to this is that modular curves such
as $X(N)$ and $X_0(N)$ have a rich structure in arithmetic geometry,
so that rational points on these curves (over a number field $K$)
correspond to interesing elliptic curves defined over $K$.  Having access
to good models of modular curves is also useful in a number of algorithmic
applications, such as the Schoof--Elkies--Atkin algorithm for counting points
on an elliptic curve over a finite field.  The supreme arithmetic
application of modular curves is in their relation to the Galois
representations attached to Hecke eigenforms, as extensively illustrated in
other lectures from this summer school.  For a cuspidal eigenform $f$ of
weight~$2$ on $\Gamma_1(N)$, say, the mod~$\ell$ Galois representation
$\rho_{f,\ell}$ can be realized inside the $\ell$-torsion points of the
Jacobian variety of the modular curve $X_1(N)$, and the $\ell$-adic Galois
representation can be assembled out of the $\ell^n$ torsion points of the
Jacobian, for varying~$n$.  It is thus of interest to be able to find
explicit algebraic equations for modular curves and their Jacobians.

These explicit equations can even help with finding explicit models
for the Galois representations $\rho_{f,\ell}$ if the weight of $f$ is
greater than~$2$: in that case, the Galois representation is realized
in an \'etale cohomology group of a modular curve with respect to a
nonconstant system of coefficients.  This is less amenable to direct
computation, but it turns out that our given $f$ is in fact congruent
modulo~$\ell$ to a Hecke eigenform $g$ of weight~$2$ but of level
$\Gamma_1(N\ell)$, so that the mod~$\ell$ representations of $f$ and $g$
are the same.  So, subject to increasing the level, this reinforces the
usefulness of having access to explicit models for modular curves and for
working with their Jacobians.  This approach is used in the work
of Couveignes-Edixhoven~\cite{CouveignesEdixhoven} and their students to 
give algorithms to compute explicit Galois representations attached to
forms of higher weight.

Now that the reader is, we hope, sufficiently motivated to find
models of modular curves, we address the issue of precisely how we can
represent a smooth projective algebraic curve $X$, such as $X(N)$.
Broadly speaking, one can view such a curve algebraically, via a model
for its function field $\C(X)$, or geometrically, via an embedding
of the curve $X$ in some projective (or other explicit) space
$\Projective^n$.  From the algebraic point of view, the field $\C(X)$
is of transcendence degree~$1$ over $\C$, so one chooses a
transcendental element $x \in \C(X)$, and considers the finite
extension $\C(X)$ of $\C(x)$.  This is the same as considering a
finite map of curves $X \to \Projective^1$, where $x$ is the
coordinate on $\Projective^1$, so $\C(\Projective^1) = \C(x)$.  Let
$y \in \C(X)$ be a primitive element for the field extension, so
$\C(X) = \C(\Projective^1)[y]$.  Then the elements $x,y$ generate
the function field $\C(X)$, and they satisfy a polynomial equation
$f(x,y)=0$.  The geometric meaning of 
this is that $X$ is birationally equivalent to the plane curve with
affine equation $f(x,y)=0$.  However, the plane curve in question will
usually have singularities (including at infinity, once one moves to
the projective plane), and working directly with that plane model can
be delicate.  What one usually does is to work with the extension
$\C(X)/\C(\Projective^1)$ using algorithms analogous to those used for
computing in a number field $\Q(\alpha)/\Q$: there, one computes the integral
closure $R$ of $\Z$ in $\Q(\alpha)$ to find the ring of integers, and
represents fractional ideals of $R$ as $\Z$-lattices of rank
$[\Q(\alpha):\Q]$.  In the function field case, one has to consider
integral closures over both $\C[x]$ (the analog of $\Z$ here) and a
ring such as $\C[1/x]$, in order to get a handle on the points of $X$
lying above $\infty \in \Projective^1$.  In this lecture, we will use
the algebraic point of view to describe a model for $X_0(N)$, where
the map to $\Projective^1$ is the natural projection to $X(1)$, the
transcendental element generating $\C(X(1))$ is the usual
$j$-function, and the polynomial that we called $f(x,y)$ above is in
fact the modular polynomial $\Phi_N(j,j')$.

As for the geometric point of view, one can run the range between two
extremes.  On the one hand, one can represent $X$ as a curve in
$\Projective^2$ or $\Projective^3$, so the equations of $X$ involve
few variables, but can be of high degree.  On the other hand, one can
take an embedding of $X$ arising from a line bundle of moderately
large degree; this yields an embedding of $X$ into a projective space
$\Projective^n$ with $n$ moderately large (but still comparable to the
genus $g$ of $X$), however with the benefit that the equations for $X$
now have low degree and a simpler structure.  We will illustrate this
second approach later for $X(N)$, where the line bundle in question is
the one whose sections are modular forms of a given weight on
$\GammaN$.  That will require us to review a few constructions in
algebraic geometry, so we will postpone it, and start with the more
concrete approach of using the modular equation to get models for
$X_0(N)$.

We thus proceed to study the modular curve $X_0(N)$, which
parametrizes pairs of elliptic curves connected by an isogeny whose
kernel is cyclic of order $N$.  Over $\C$, one can always analytically
bring this situation to the map $\C/(\Z + \Z N\tau) \to \C/(\Z + \Z\tau)$.
Equivalently, there are two maps from $X_0(N)$ to $X(1)$,
the first sending $\tau \in \Gamma_0(N) \backslash \Half$ to $\tau \in
\Gamma(1)\backslash \Half$, and the second sending $\tau$ to $N\tau
\in \Gamma(1)\backslash \Half$.  The resulting map $X_0(N) \to
X(1)\times X(1)$, induced by the map $\tau \mapsto (\tau, N\tau)$ from
$\Half$ to $\Half \times \Half$, is a birational equivalence between
$X_0(N)$ and its image.  Concretely, we can use the $j$-function as a
coordinate on $X(1)$ to identify $X(1)$ with $\Projective^1$.  We then
obtain that the function field of $X_0(N)$ is generated by the two
modular functions $j(\tau)$ and $j' = j(N\tau)$.  These play the role
of the elements $x,y$ in our above discussion of function fields of
general curves.  We thus have the following classical result.

\begin{proposition}
\label{proposition2a.1}
The function field $\C(X_0(N))$ is generated by the two elements
$j(\tau)$ and $j(N\tau)$, with a single polynomial relation between
them, called the modular equation:
\begin{equation}
\label{equation2a.1}
\Phi_N(j(\tau),j(N\tau)) = 0, \qquad \Phi_N(x,y) \in \Z[x,y].
\end{equation}
The polynomial $\Phi_N(x,y)$ is called the $N$th modular polynomial;
we assert that its coefficients actually belong to $\Z$.  It has the
property that if $E$ and $E'$ are elliptic curves (a priori, over $\C$, but
this works more generally) with $j$-invariants $j(E)$ and $j(E')$,
then there exists a cyclic $N$-isogeny between $E$ and $E'$ if and
only if $\Phi_N(j(E), j(E')) = 0$.
\end{proposition}
For the proof, see for example Chapter~5 of~\cite{LangElliptic}.  Let
us sketch in these notes one way to compute the 
polynomial $\Phi_N(x,y) = \sum_{k,\ell} c_{k,\ell}x^k y^\ell$.
Similarly to the situation with Hecke operators, one considers the 
decomposition of a double coset into single cosets, where we know each
$\Gamma(1)$-coset contains an upper triangular representative:
\begin{equation}
\label{equation2a.2}
\Gamma(1) \twomatr{N}{}{}{1} \Gamma(1) 
    = \bigsqcup_{\text{certain } a,b,d} \Gamma(1) \twomatr{a}{b}{}{d}.
\end{equation}
This means effectively that if one fixes $\tau$ and hence a value
$j(\tau)$, then the roots in $y$ of the polynomial $\Phi_N(j(\tau),y)$
are the values $y = j((a\tau+b)/d)$ for those $(a,b,d)$ that appear
(parametrizing different sublattices of $\Z+\Z\tau$ of cyclic index
$N$); note that one of the values of $(a,b,d)$ is $(N,0,1)$, which
corresponds to the root $y = j(N\tau)$.  We thus conclude that
\begin{equation}
\label{equation2a.3}
\sum_{k,\ell} c_{k,\ell} \, j(\tau)^k y^\ell
  = \Phi_N(j(\tau),y)
   = \prod_{\text{the same } a,b,d}
         \bigl[y - j((a\tau+b)/d)\bigr].
\end{equation}
It follows that the coefficients $c_{k,\ell}$ for fixed $\ell$ and
varying $k$ are obtained when one expresses the $\ell$th symmetric
polynomial in the $\{j((a\tau+b)/d)\}$ as a polynomial in
$j(\tau)$.  This is possible because this symmetric polynomial is a
modular function that is $\Gamma(1)$ invariant (due to the double
coset in~\eqref{equation2a.2}), and its only pole is at the
cusp~$\infty$.  We can compute the $q$-expansion of this $\ell$th
symmetric polynomial from the $q$-expansion of $j(\tau) = q^{-1} + 744 + \cdots
\in \Z[[q]]$, and then identify the resulting series in $q$ as a
polynomial $\sum_k c_{k,\ell} \, j(\tau)^k$.  In carrying out this
calculation, one uses $j((a\tau+b)/d) = \zeta_d^{-b} q^{-a/d} + 744 + \cdots$,
where $\zeta_d = \exp((2\pi i)/d)$.  In fact, the calculation takes
place over $\Z[\zeta_N]$, which contains all the $\zeta_d$.  The
invariance of everything under $\Gamma(1)$ implies however that the final
result is invariant under any Galois automorphism of $\Q(\zeta_N)$,
which shows that the coefficients in the final answer all belong
to~$\Z$. 

The coefficients of $\Phi_N$ are notoriously large, and the
birational plane model for $X_0(N)$ given by the equation
$\Phi_N(j,j') = 0$ is rather singular, but this model is still quite
useful in explicit computations.  It should be pointed out that there
are now better ways to compute the modular polynomial, namely, by
interpolation.  The degree of $\Phi_N$ is known (e.g., for $N$ prime,
it is $N+1$), and one knows that $\Phi_N(x,y) = \Phi_N(y,x)$, because
the dual morphism to a cyclic $N$-isogeny is again cyclic of
degree~$N$.  It follows that it is enough to generate enough points
$(j_\alpha,j'_\alpha)$ on the curve $\Phi_N(j,j') = 0$, in order to
obtain enough values to solve for the coefficients
$c_{k,\ell} = c_{\ell,k}$.  The articles \cite{Enge}
and~\cite{BrokerLauterSutherland} do this respectively for collections
of points $(j_\alpha,j'_\alpha) \in \C^2$ or 
$(j_\alpha,j'_\alpha) \in \overline{\F}_p^2$, by taking
a suitable collection of isogenous pairs of elliptic curves over
$\C$ or $\overline{\F}_p$.  In the latter setting, one gets equations for
the $c_{k,\ell} \bmod p$, which one can combine for various~$p$ to
obtain the true value over~$\Z$.

We now move on to the second approach outlined in the introduction of
finding equations for modular curves.  As promised, we begin with an
informal overview of the needed prerequisites from algebraic geometry:
line bundles on (as always, smooth projective) algebraic curves.
For pedagogical reasons, we continue to work over $\C$, to allow the
reader to visualize the situation in the analytic category, not just
algebraically.

\begin{definition}
\label{definition2.1}
A complex line bundle $\LL$ on an algebraic curve $X$ is a
choice, for each point $p \in X$, of a one-dimensional complex vector
space $\LL_p$, in a way that ``varies holomorphically'' with $p$.

Concretely, this means that one can cover $X$ by open sets $\{U_i\}$
such that, for each $U_i$, the totality of vector spaces $\LL_{U_i} =
\{\LL_p \mid p \in U_i\}$ is isomorphic to the product $U_i \times
\C$.  This means that there is an isomorphism (of two-dimensional
complex manifolds) $\psi_i: \LL_{U_i} \to U_i \times \C$, where a
vector $v \in \LL_p$ is mapped to $\psi_i(v) = (p,c)$ for some $c\in
\C$, and the map sending $v$ to $c$ is a $\C$-linear isomorphism
between $\LL_p$ and $\C$; hence the set $\LL_p \subset \LL_{U_i}$ is
identified with $\{p\} \times \C \subset U_i \times \C$.
Whenever $U_i \intersect U_j \neq \emptyset$, these two different
identifications of $\LL_p$ can be compared via a homomorphic nowhere
vanishing transition function $\phee_{i,j}: U_i \intersect U_j \to
\C^*$, where $\psi_j(\psi_i^{-1}(p,c)) = (p,\phee_{i,j}(p) c)$.

Conversely, given a covering $\{U_i\}$ of $X$ by open sets, and a
collection of transition 
functions $\phee_{i,j}$ (which need to be compatible on intersections
$U_i \intersect U_j \intersect U_k$), then one can glue the line
bundles $\{U_i \times \C\}$ together, using the $\phee_{i,j}$, to
obtain a line bundle $\LL$ on $X$.
\end{definition}

The key concept that will matter to us is that of a holomorphic
section of a line bundle $\LL$ on $X$.  This generalizes holomorphic
functions on $X$, which are sections of the trivial line bundle
$X \times \C$.

\begin{definition}
\label{definition2.2}
Let $\LL$ be a line bundle on $X$.  A (holomorphic) section $s$ of
$\LL$ is a function $s:X \to \LL$, such that for every $p\in X$, we
have $s(p) \in \LL_p$.  In terms of the local isomorphisms
$\{\psi_i: \LL_{U_i} \to U_i \times \C\}$, requiring $s$ to be holomorphic
means that for $p \in U_i$, $\psi_i(s(p)) = (p,f_i(p))$ with
$f_i:U_i \to \C$ a holomorphic function.  The resulting ``values in
local coordinates'' $f_i$ of the section~$s$ will then be compatible
in the sense that for $p \in U_i \intersect U_j$, we have 
$f_j(p) = \phee_{i,j}(p)f_i(p)$. 

The set of holomorphic sections is written $H^0(X,\LL)$; it is a
finite-dimensional vector space, that we can always identify with a
Riemann-Roch space, as we will discuss presently.  We can also
consider meromorphic sections of $\LL$, which the reader should have
no trouble defining.  Although $H^0(X,\LL)$ can be zero, there are
always nonzero meromorphic sections of $\LL$.
\end{definition}

We now describe the relation with Riemann-Roch spaces.  Recall that
for a divisor $D = \sum n_p p$ on $X$, the Riemann-Roch space $L(D)$
is the set of function field elements $f \in \C(X)$ that satisfy
$\Divisor f + D \geq 0$; in other words, for each of the (finitely
many) $p$ in the support of $D$, we have $v_p(f) \geq -n_p = -v_p(D)$.
Here the valuation $v_p$ gives the order of the zero of $f$ at $p$ (or
of the pole, if $v_p(f) < 0$); for convenience, we set $v_p(0) = +\infty$.
We remark that the valuation $v_p$ also makes sense for a meromorphic
section $s$ of a line bundle $\LL$, where it will be written
$v_{p.\LL}(s)$.  Namely, suppose $p \in U_i$ for one of the open sets
of the cover, where $s$ is represented by the function $f_i$.  Then
$v_{p,\LL}(s) = v_p(f_i)$.  This is independent of the choice of $U_i$
containing $p$, since $\phee_{i,j}(p) \neq 0$ whenever
$p \in U_i \intersect U_j$.

\begin{proposition}
\label{proposition2.3}
Let $D$ be a divisor on $X$.  Then there exists a line bundle $\LL_D$
with the property that a meromorphic section $s$ of $\LL_D$ can 
be identified with a meromorphic function $\phi_s \in \C(X)$ on
$X$, but with a modified valuation:
$v_{p,\LL_D}(s) = v_p(\phi_s) + v_p(D)$.  Thus $s \in H^0(X,\LL_D)$ if and
only  if for every $p\in X$, we have $v_{p,\LL_D}(s) \geq 0$, which
corresponds precisely to $\phi_s \in L(D)$; note that it is possible to
have $H^0(X,\LL_D) = L(D) = 0$.

Conversely, every line bundle $\LL$ on $X$ is isomorphic to a line
bundle $\LL_D$ for some $D$, which is unique up to equivalence of
divisors (by principal divisors of rational functions in $\C(X)$).
\end{proposition}
\begin{proof}
Every divisor $D$ is locally principal, in the sense that there exists
an open cover $\{U_i\}$ of $X$ (in either the analytic or Zariski
topology), with a nonzero meromorphic function $u_i$ on each $U_i$ 
satisfying $(\Divisor u_i)|_{U_i} = D|_{U_i}$.  (The restriction
$D|_{U_i}$ of a divisor $D$ can be thought of as the intersection
$D \intersect U_i$, i.e., the restriction includes only those points
of $D$ that belong to $U_i$.)  Then construct $\LL_D$ by gluing the
$U_i \times \C$ along the transition functions $\phee_{i,j} =
u_j/u_i$.  A holomorphic (respectively, meromorphic) section $s$ of
$\LL_D$ thus corresponds to a collection $\{f_i\}$ of holomorphic
(respectively, meromorphic) functions on each $U_i$, satisfying 
$f_j = (u_j/u_i) f_i$.  Every section $s$ corresponds to the unique
$\phi_s$ that is obtained by gluing together the functions $f_i/u_i$.
So locally, $f_i = \phi_s \cdot u_i$.  Recall that for $p\in U_i$,
we have $v_{p,\LL_D}(s) = v_p(f_i)$; this yields the desired relation
between the valuations of $s$ and $\phi_s$.

For the converse, let $\LL$ be given, choose any nonzero meromorphic
section $s_0$ of $\LL$, and let
$D = \sum_p v_{p,\LL}(s_0)\cdot p = \Divisor_{(\LL)} s_0$ be the
divisor of $s_0$, viewed as a section of $\LL$.  Then we can identify 
any other meromorphic section $s$ of $\LL$ with the meromorphic
function $\phi_s = s/s_0 \in \C(X)$; note that although the values of
$s$ and $s_0$ at a point $p$ belong to $\LL_p$, their ratio is
canonically an element of $\C$ (at least, away from the poles of
$\phi_s$).  This identifies $\LL$ with $\LL_D$; incidentally, the
section $s_0$ of $\LL$ corresponds to the collection of functions
$\{f_i\} = \{u_i\}$ which give a section of $\LL_D$.  In terms of
$\LL$, the bijection between $L(D)$ and $H^0(X,\LL)$ identifies
$\phi \in L(D)$ with $\phi \cdot s_0 \in H^0(X,\LL)$.  Finally, if we
make a different choice of meromorphic section $s_1$ instead of $s_0$
at the start, this modifies $D$ by the principal divisor
$\Divisor(s_1/s_0)$.
\end{proof}

The next important notion in our overview is the degree of a line bundle.

\begin{definition}
\label{definition2.4}
Let $\LL$ be a line bundle on the Riemann surface $X$.  We say that
$\deg \LL = d$ if one meromorphic section $s$ of $\LL$ vanishes at
exactly $d$ points, counting multiplicities, and subtracting any
multiplicities of poles.  Thus $\deg \LL = \deg \Divisor_{(\LL)} s
=  \sum_p v_{p,\LL}(s)$, and this degree does not depend on the choice
of $s$, since all other choices are of the form $sf$ with $f \in
\C(X)$, with moreover $\deg \Divisor f = 0$.

Equivalently, if $\LL \isomorphic \LL_D$, then $\deg \LL = \deg D$.
\end{definition}

A basic consequence of Riemann-Roch is that if $X$ has
genus $g$, and $\deg \LL \geq 2g-1$, then
$\dim H^0(X,\LL) = \deg \LL + 1 - g$.  Another consequence is that if
$\deg \LL \geq 2g$, then $\LL$ is base point free, which means that
for every $p\in X$, there exists a holomorphic section $s \in
H^0(X,\LL)$ with $s(p) \neq 0$.

We are now ready to discuss some aspects of the relation between line
bundles on a curve $X$, and maps from $X$ to a projective space.

\begin{definition}
\label{definition2.5}
Let $\LL$ be a base point free line bundle on $X$.  Take a basis
$\{s_0, s_1, \dots, s_n\}$ for $H^0(X,\LL)$ (more generally, we only
need a basis for a base point free subspace of $H^0(X,\LL)$).  The
associated map from $X$ to the projective space $\Projective^n$ is
given by
\begin{equation}
\label{equation2.1}
\phee:X \to \Projective^n,
\qquad
\phee(p) = [s_0(p):s_1(p):\cdots:s_n(p)].
\end{equation}
Note in the above that, as usual, the values $s_i(p)$ all belong to
the one-dimensional vector space $\LL_p$, but that the proportions
between their values make enough sense for us to get the projective
coordinates of a point in $\Projective^n$.  The reason for requiring
$\LL$ to be base point free is to ensure that we never map a point $p$
to the invalid projective point $[0:0:\cdots:0]$.
\end{definition}

\begin{example}
\label{example2.6}
Let $X$ be an elliptic curve, say for definiteness with affine equation 
$y^2 = x^3 + 3141x + 5926$, and let $O \in X$ be the point at infinity.
Consider the line bundles $\LL_{3O}$ and $\LL_{4O}$.  We can identify
$H^0(X,\LL_{3O})$ with the Riemann-Roch space $L_{3O}$, which has the
basis $\{1,x,y\}$.  The resulting map from $X$ to the projective plane
is the usual one; it sends the affine point $p$ to the projective
point $[1:x(p):y(p)]$, while the point $O$ is sent to $[0:0:1]$.  One
can see this by ``continuity'', because of the Laurent series $x =
t^{-2} + \cdots$ and $y = t^{-3} + \cdots$ in terms of a uniformizer
$t$ at $O$, so as our point ``approaches'' $O$, its projective
coordinates 
$[1:t^{-2}+\cdots:t^{-3}+\cdots] = [t^3: t + \cdots: 1 + \cdots]$
``approach'' $[0:0:1]$.  A less informal way to see this is to
remember that the sections in $H^0(X,\LL_{3O})$ corresponding to
$1,x,y$ are in fact everywhere holomorphic, when viewed as sections of
the line bundle, and to work with a trivialization of $\LL_{3O}$ in
a neighborhood of $O$; this corresponds to the transition function
$y^{-1}$ sending the local coordinate functions $1,x,y$ to $1/y,x/y,1$
near $O$.

As for the line bundle $\LL_{4O}$, take the basis
$\{s_0,s_1,s_2,s_3\}$ of $H^0(X,\LL_{4O})$, corresponding to the basis
$\{1,x,y,x^2\}$ of $L(4O)$.  The resulting map $p \mapsto
[s_0(p):s_1(p):s_2(p):s_3(p)] \in \Projective^3$ embeds $X$ as the
intersection of the two quadric surfaces $s_1^2 - s_0 s_3 = 0$ and
$s_2^2 - s_1s_3 - 3141s_0s_1 - 5926s_0^2 = 0$. 
\end{example}

In the above example, the image of the genus~$1$ curve $X$ under the
projective embedding given by $\LL_{4O}$ is described by quadrics
(i.e., by polynomials of degree~$2$).  This is a special case of the
following general theorem, due independently to Fujita~\cite{Fujita}
and Saint-Donat~\cite{SaintDonat1,SaintDonat2}, building on
results of Castelnuovo and Mumford: 

\begin{theorem}
\label{theorem2.7}
If $X$ has genus~$g$, and $\deg \LL \geq 2g+2$, then the map to
projective space given by $\LL$ is an embedding of $X$, and the image
is defined by quadrics; more precisely, the homogeneous ideal of
vanishing of the image of $X$ in projective space is generated by its
degree~$2$ elements.
\end{theorem}

We now finally come to the application of all this to modular curves.
We first review how modular forms of weight~$k$ on $\GammaN$, with $N
\geq 3$, are sections of a particular line bundle $\LL_k$ on $X(N)$;
this result holds in fact for any subgroup of $\Gamma(1)$, but the
advantage of the group $\GammaN$ is that it has no elliptic points for
$N \geq 3$, and all its cusps are moreover regular; it follows that 
$\LL_k \isomorphic \LL_1^{\otimes k}$, so that
$\deg \LL_k = k \deg \LL_1$.  In the presence of elliptic points or
irregular cusps, the degree of $\LL_k$ is slightly more delicate; see for
example the discussion of the divisor of a modular form in Chapter~2
of~\cite{ShimuraBook}.

To define the line bundle $\LL_k$ on $X(N)$, we depart from our
previous description in terms of an open cover, and instead obtain
$\LL_k$ as the quotient of the trivial bundle on $\Half$ by a
nontrivial action of $\GammaN$.  To be precise, we need to consider the
extended upper half plane $\Half^* = \Half \union \Q \union\{\infty\}$,
with the
topology given in Chapter~I of~\cite{ShimuraBook}, so as to correctly
deal with the cusps; we will however ask for the reader's indulgence,
and gloss over this important point from here on.
The idea is that (holomorphic) sections of the trivial bundle $\Half
\times \C$ are precisely holomorphic functions $f: \Half \to \C$.  We
define an action of $\Gamma(1)$ on the line bundle $\Half \times \C$,
in such a way that sections invariant under a subgroup $\Gamma$ of
$\Gamma(1)$ are precisely the modular forms of weight~$k$ on $\Gamma$.
One can then see that the desired action of 
$\gamma = \stwomatr{a}{b}{c}{d} \in \Gamma(1)$ on a pair
$(\tau,z) \in \Half  \times \C$ is given by
\begin{equation}
\label{equation2.2}
\gamma \cdot (\tau, z)
   = \Bigl(\frac{a\tau+b}{c\tau+d},
        (c\tau + d)^k\cdot z \Bigr).
\end{equation}
We then define $\LL_k$ to be the resulting line bundle on
$X(N) = \GammaN \backslash \Half$ (where we apologize one last time
about the cusps) whose total space is $\GammaN\backslash (\Half \times \C)$.
It follows that
\begin{equation}
\label{equation2.3}
H^0(X(N),\LL_k) = \modforms_k(\Gamma(N));
\end{equation}
it turns out to be slightly more convenient to use the full space
of modular forms than to restrict to cusp forms, which would
correspond to sections of $\LL_k$ that vanish at all cusps.

We can now use a basis $\{f_0, \dots, f_n\}$ for
$\modforms_k(\Gamma(N))$ to obtain a projective embedding of $X(N)$
into $\Projective^n$, at least when $\deg \LL_k \geq 2g+2$, with $g$ the
genus of $X(N)$.  Similarly to~\eqref{equation2.1}, this sends $\tau
\in \GammaN \backslash \Half$ to $[f_0(\tau):\cdots:f_n(\tau)] \in
\Projective^n$; the value in projective space is independent of the
representative $\tau$ chosen, provided all the $f_i$ are evaluated at
the same $\tau$.

\begin{proposition}
\label{proposition2.8}
Let $N \geq 3$.  For the resulting curve $X(N)$, the line bundle 
$\LL_2$ has degree $\deg \LL_2 \geq 2g+2$, and hence gives rise to a
projective embedding of the modular curve with image described by
quadrics.  Knowing the equations of the resulting modular curve is
equivalent to knowing the multiplication map $\modforms_2(\GammaN) \times
\modforms_2(\GammaN) \to \modforms_4(\GammaN)$.
\end{proposition}
\begin{proof}
We can relate $\LL_2$ to the canonical line bundle $\Omega^1$ of
holomorphic $1$-forms on $X(N)$.  It is standard that 
$H^0(X(N),\Omega^1) \isomorphic \cuspforms_2(\GammaN)$ (in fact, this works
for any group $\Gamma$), by identifying a cusp form $f(z)$ with the
differential form $f(z) dz$, which is now invariant under $\GammaN$.  The
reason that the corresponding modular forms are cuspidal can be seen
in terms of the local coordinate at infinity $q^{1/N} = \exp(2 \pi i
z/N)$.  Since $2 \pi i N^{-1} \, dz = q^{-1/N} \, d(q^{1/N})$, the
$q$-expansion of $f$ must start 
with $c_1 q^{1/N} + c_2 q^{2/N} + \cdots$ for $f(z) dz$ to avoid a
pole at $q=0$, i.e., at $z = \infty$; a similar condition holds at the
other cusps.  A 
modular form 
which does not have to be cuspidal thus corresponds to a meromorphic
section of $\Omega^1$, with a possible simple pole at each cusp 
of $X(N)$.  This means that $\LL_2 \isomorphic \Omega^1 \tensor
\LL_{\textbf{cusps}}$, where $\textbf{cusps}$ is the divisor of
the cusps.  In particular,
\begin{equation}
\label{equation2.4}
\deg \LL_2 = \deg \Omega^1 + \deg(\textbf{cusps}) = (2g-2) + c,
\end{equation}
where $c$ is the number of cusps of $X(N)$.  But one knows
that $c \geq 4$ 
once $N \geq 3$.  Thus $\deg \LL_2 \geq 2g+2$, as desired, and the
projective embedding given by a basis of $\modforms_2(\GammaN)$ is
described by quadrics.

Let $\{f_0, \dots, f_n\}$ be a basis for $\modforms_2(\GammaN)$.  Suppose
we wish to determine the quadrics that vanish on the image of $X(N)$, since
these generate the homogeneous ideal of the projective curve.  The presence
of such a quadric of the form
$q(T_0, \dots, T_n) = \sum_{i,j} c_{i,j} T_i T_j$, in 
terms of the homogeneous coordinates $[T_0:\cdots:T_n]$, corresponds to the
identity of modular forms
$\sum_{i,j} c_{i,j} f_i f_j = 0 \in \modforms_4(\GammaN)$.  Hence, to find
the generators of our homogeneous ideal, it is enough to know how to
multiply $f_i f_j$ for every pair $(i,j)$, and how to find linear relations
between these elements of $\modforms_4(\GammaN)$.
\end{proof}

We remark that the same result holds for $X_1(N)$ with $N \geq 5$.  Let us
therefore compute equations for $X_1(5)$ as an example.  This is not
extremely interesting, since the genus of $X_1(5)$ is zero, but it
illustrates the above theorem.  In this situation, we have
$\dim \modforms_2(\Gamma_1(5)) = 3$, and a basis is $\{f,g,h\}$ with
\begin{equation}
\label{equation2.5}
\begin{split}
f &= 1 \quad\quad + 60q^3 - 120q^4 + \cdots,
\\
g &= \quad q \quad + 6q^3 - 9q^4 + \cdots,
\\
h &= \quad \>\>\> q^2 - 4q^3 + 12q^4 + \cdots.
\\
\end{split}
\end{equation}
One also knows that $\modforms_4(\Gamma_1(5))$ is $5$-dimensional, and its
elements are determined by knowing their $q$-expansions up to and including
the $q^4$ term.  One then computes $f^2 = 1 + 120q^3 - 240 q^4 + \cdots$,
$fg = q + 6q^3 + 51q^4 +\cdots$, and so forth, all of which belong to 
$\modforms_4(\Gamma_1(5))$.  One then obtains the equation
\begin{equation}
\label{equation2.6}
  g^2 - fh - 4gh - 16h^2 = 0,
\end{equation}
which means that we have identified $X_1(5)$ with the conic in the
projective plane given by the equation $U^2 - TV - 4UV - 16V^2 = 0$.  The
conic contains the rational point $[T:U:V] = [1:0:0]$ (namely, the point
$q=0$ corresponding to the cusp $\infty$), so we can identify 
the curve $X_1(5)$ with $\Projective^1$ over $\Q$, not just over $\C$.

We point out that describing the multiplication map 
$\modforms_2(\GammaN) \times \modforms_2(\GammaN) \to \modforms_4(\GammaN)$
can be done by interpolation, which ties in with our earlier description of
finding the modular equation $\Phi_N$ by interpolation.  Namely, suppose
that we take a large number (``$L$'') of points $\tau_1, \dots, \tau_L \in
X(N)$, where we require $L > \deg \LL_4$.  Then a modular form 
$f \in \modforms_4(\GammaN)$ is completely determined by its values at
these $L$ points, because if $g$ were a different form agreeing with $f$ at
$\tau_1, \dots, \tau_L$, then the difference $f-g$ would be a nonzero
section of the line bundle $\LL_4$, so $f-g$ could only vanish at 
$\deg \LL_4$ points, contradicting the fact that it vanishes at 
$\tau_1, \dots, \tau_L$.  In that case, the basis $\{f_0, \dots, f_n\}$ of  
$\modforms_2(\GammaN)$, as well as all products $f_i f_j$, can be represented
by their values at $\tau_1, \dots, \tau_L$.  Thus we can identify $f_i$ by
its vector of values $(f_i(\tau_1), \dots, f_i(\tau_L))$, and carry out
multiplication into $\modforms_4$ componentwise in order to find the
quadrics that vanish on the image of $X(N)$.  This is equivalent to finding
the quadrics on $\Projective^n$ that vanish on (i.e., interpolate through)
the projective points $P_1, \dots, P_L$, which are the images of $\tau_1,
\dots, \tau_L$ under the projective embedding.  Concretely,
\begin{equation}
\label{equation2.7}
P_j = [f_0(\tau_j) : \cdots : f_n(\tau_j)]
\end{equation}
so we have reversed our viewpoint.  Whereas we previously fixed $f_i$ and
represented it as a function by its values at varying $\tau_j$, we now
fix $\tau_j$ and represent it as a projective point by the values of the
various $f_i$ at that point.  It turns out~\cite{KKMjacobians} that this is 
an effective way to represent the curve if one is interested in computing
with its Jacobian; this is called ``Representation~B'' in that article.  We
also note that the points $\tau_1,\dots,\tau_L$ do not actually have to be
distinct; representing modular forms by their $q$-expansions up to degree
$q^L$, as we did in the example above, is a way of evaluating the forms at
the divisor $L\cdot \infty$.  We note that the family of Katz modular forms
from Lecture~1 gives an ample supply of modular forms that can be easily
evaluated at points, since evaluating at a $\tau_j$ can be carried out
algebraically by evaluating on a tuple $(E,P,Q,\omega)$.

\section{Exercises}

The following exercises were distributed to students at
the summer school.

\subsection*{Exercise 0}
Give the argument 
that evaluating a weight $k$ Katz modular form on the tuple 
$(E_\tau, 1/N, \tau/N, dz)$ defines a function $f(\tau)$ that
transforms like a usual weight $k$ modular form.

\subsection*{Exercise 1}
a) Use SageMath or Magma (or any other software) to find the modular
polynomial $\Phi_2(X,Y)$ from the identity 
$$\Phi_2(X,j(\tau)) = (X-j(2\tau))(X - j(\tau/2))(X - j((\tau+1)/2)),$$
by comparing $q$-expansions.

b) Find $\Phi_2(X,Y)$ in a different way, by finding enough pairs
$(j,j')$ of $j$-invariants of elliptic curves that are $2$-isogenous
to interpolate $\Phi_2$ through these points.  (What is the degree of
$\Phi_2$ in each of $X$ and $Y$, and how many points are needed?  I
suggest taking curves $E: y^2 = x(x-1)(x-\lambda)$ for a few
values of $\lambda \in \Q$, and all their quotients by cyclic
$2$-torsion subgroups.)

c) Question to think about later: can you find, e.g., $\Phi_5(X,Y)$ by
finding its reduction modulo many primes?  This involves finding for
each $p$ a number of pairs of $5$-isogenous elliptic curves over
$\F_p$ and interpolating through the corresponding $(j,j')\bmod p$.





\subsection*{Exercise 2}
In this exercise, you may assume that $N \geq 3$ is prime if you like,
but see if you can do the general case too.

a) What is the index $[\Gamma(1):\Gamma(N)]$?  What is the degree $d$ of
the map $\pi: X(N) \to X(1)$ between modular curves?

b) What is the ramification of $\pi$ at the cusps?  Use this to find
the number $c$ of cusps of $X(N)$.

c) Let $\LL$ be the line bundle on $X(N)$ whose sections give
$\modforms_1(\Gamma(N))$.  Show that $\deg \LL = d/12$.  (Hint:
$\Delta(z) \in \cuspforms_{12}(\Gamma(1)) \subset \cuspforms_{12}(\Gamma(N))$.)

d) Find the genus of $X(N)$ in terms of $d$ and $c$.  (Hint: consider the
line bundles $\LL^2$, whose sections are $\modforms_2(\Gamma(N))$, and
$\Omega^1 \isomorphic \LL^2(-\textbf{cusps})$, whose sections are
$\cuspforms_2(\Gamma(N))$.)

\subsection*{Exercise 3}
Let $X$ be the (projective model of the) curve $y^3 = x^4 + x
+ 2$.  Its points are the affine points satisfying the above equation,
plus one point $P_0$ at infinity where the rational function $x$ has a
pole of order $3$, and $y$ has a pole 
of order $4$.  This is a $C_{3,4}$ curve; generally, $C_{a,b}$ curves
are a nice source of examples.  (However, if $\abs{a - b} \geq 2$, then
the plane model of a $C_{a,b}$ curve is singular at infinity.) 

a) Compute the Riemann-Roch spaces $L(k P_0)$ for $k \leq 15$, and
deduce that $X$ has genus $3$, either from Riemann-Roch or by any other
method you like.

b) For each of $k = 6,7,8$, consider the resulting map of $X$ to
projective space, and give generators for the ideal describing the
image.

\subsection*{Exercise 4}
Our goal is to find equations for $X = X_1(11)$, which has genus 
$1$.  

a) Use SageMath or Magma to find $q$-expansions of a basis for
$\modforms_2(\Gamma_1(11))$, ordered as an echelon basis in terms of the
order of vanishing at the cusp $\infty$ (i.e., $q=0$).

b) The projective embedding of $X$ given by $\modforms_2(\Gamma_1(11))$
has too large dimension for a human-readable model of the curve.
Instead, obtain a smaller embedding by restricting to a subset 
$V \subset \modforms_2(\Gamma_1(11))$ defined by imposing a certain order of
vanishing at the cusp $\infty$.  This means
that, viewing $\modforms_2(\Gamma_1(11)) = H^0(X,\LL)$ for a suitable line
bundle $\LL$ on $X$, your space $V$ will be $H^0(X, \LL(-k\infty))$ for
some $k$.  Thus the line bundle you will consider will have degree
$(\deg \LL) - k$.

Suggestion: Take $(\deg \LL) - k = 3$ or $4$.  This will produce
either one cubic equation in $\Projective^2$, or two quadric equations in
$\Projective^3$.

\bibliographystyle{amsalpha}
\bibliography{sarajevo2}

\providecommand{\bysame}{\leavevmode\hbox to3em{\hrulefill}\thinspace}
\providecommand{\MR}{\relax\ifhmode\unskip\space\fi MR }
\providecommand{\MRhref}[2]{%
  \href{http://www.ams.org/mathscinet-getitem?mr=#1}{#2}
}
\providecommand{\href}[2]{#2}
\begin{thebibliography}{BLS12}

\bibitem[BLS12]{BrokerLauterSutherland}
Reinier Br{\"o}ker, Kristin Lauter, and Andrew~V. Sutherland, \emph{Modular
  polynomials via isogeny volcanoes}, Math. Comp. \textbf{81} (2012), no.~278,
  1201--1231. \MR{2869057}

\bibitem[EC11]{CouveignesEdixhoven}
Bas Edixhoven and Jean-Marc Couveignes (eds.), \emph{Computational aspects of
  modular forms and {G}alois representations}, Annals of Mathematics Studies,
  vol. 176, Princeton University Press, Princeton, NJ, 2011, How one can
  compute in polynomial time the value of Ramanujan's tau at a prime.
  \MR{2849700}

\bibitem[Eng09]{Enge}
Andreas Enge, \emph{Computing modular polynomials in quasi-linear time}, Math.
  Comp. \textbf{78} (2009), no.~267, 1809--1824. \MR{2501077}

\bibitem[Fuj77]{Fujita}
T.~Fujita, \emph{Defining equations for certain types of polarized varieties},
  Complex analysis and algebraic geometry, Iwanami Shoten, Tokyo, 1977,
  pp.~165--173. \MR{0437533}

\bibitem[Hec27]{Hecke27}
E.~Hecke, \emph{Theorie der {E}isensteinschen {R}eihen h\"oherer {S}tufe und
  ihre {A}nwendung auf {F}unktionentheorie und {A}rithmetik}, Abh. Math. Sem.
  Univ. Hamburg \textbf{5} (1927), 199--224.

\bibitem[Kat76]{Katz}
Nicholas~M. Katz, \emph{{$p$}-adic interpolation of real analytic {E}isenstein
  series}, Ann. of Math. (2) \textbf{104} (1976), no.~3, 459--571. \MR{0506271}

\bibitem[KM07]{KKMjacobians}
Kamal Khuri-Makdisi, \emph{Asymptotically fast group operations on {J}acobians
  of general curves}, Math. Comp. \textbf{76} (2007), no.~260, 2213--2239
  (electronic). \MR{2336292}

\bibitem[KM12]{KKMmoduli}
\bysame, \emph{Moduli interpretation of {E}isenstein series}, Int. J. Number
  Theory \textbf{8} (2012), no.~3, 715--748. \MR{2904927}

\bibitem[Lan87]{LangElliptic}
Serge Lang, \emph{Elliptic functions}, second ed., Graduate Texts in
  Mathematics, vol. 112, Springer-Verlag, New York, 1987, With an appendix by
  J. Tate. \MR{890960}

\bibitem[SD72a]{SaintDonat1}
Bernard Saint-Donat, \emph{Sur les \'equations d\'efinissant une courbe
  alg\'ebrique}, C. R. Acad. Sci. Paris S\'er. A-B \textbf{274} (1972),
  A324--A327. \MR{0289516}

\bibitem[SD72b]{SaintDonat2}
\bysame, \emph{Sur les \'equations d\'efinissant une courbe alg\'ebrique}, C.
  R. Acad. Sci. Paris S\'er. A-B \textbf{274} (1972), A487--A489. \MR{0289517}

\bibitem[Shi94]{ShimuraBook}
Goro Shimura, \emph{Introduction to the arithmetic theory of automorphic
  functions}, Publications of the Mathematical Society of Japan, vol.~11,
  Princeton University Press, Princeton, NJ, 1994, Reprint of the 1971
  original, Kan{\^o} Memorial Lectures, 1. \MR{1291394}

\end{thebibliography}

\end{document}